\documentclass[11pt,reqno]{amsart}

\usepackage{amsmath} 
\usepackage{amsfonts,amssymb,amsthm,bm}

\usepackage{enumitem}
\usepackage[utf8]{inputenc}
\usepackage[T1]{fontenc}
\usepackage{stackrel}
\usepackage{soul}
\usepackage[margin=1in]{geometry}

\usepackage[dvipsnames]{xcolor}

\usepackage{hyperref}

\DeclareMathOperator{\Bp}{B_p}
\DeclareMathOperator{\Bq}{B_q}
\DeclareMathOperator{\Binfty}{B_\infty}
\DeclareMathOperator{\Blog}{B_{\log}}

\DeclareMathOperator{\M}{M}

\DeclareMathOperator{\FW}{FW} 
\newcommand{\Chi}{\mathcal{X}} 
\let\emptyset\varnothing 

\DeclareMathOperator*{\esssup}{ess\,sup}
\DeclareMathOperator{\RHI}{RHI}

\newcommand{\N}{\mathbb{N}}

\newcommand{\T}{\mathbb{T}}
\newcommand{\R}{\mathbb{R}}

\newcommand{\D}{\mathbb{D}}

\def\XXint#1#2#3{\mkern3mu{\setbox0=\hbox{$#1{#2#3}{\int}$ }
\vcenter{\hbox{$#2#3$ }}\kern-.6\wd0}}

\def\vint_#1{\mathchoice%
        {\mathop{\kern 0.2em\vrule width 0.6em height 0.69678ex depth -0.58065ex
                \kern -0.8em \intop}\nolimits_{\kern -0.4em#1}}%
        {\mathop{\kern 0.1em\vrule width 0.5em height 0.69678ex depth -0.60387ex
                \kern -0.6em \intop}\nolimits_{#1}}%
        {\mathop{\kern 0.1em\vrule width 0.5em height 0.69678ex
            depth -0.60387ex
                \kern -0.6em \intop}\nolimits_{#1}}%
        {\mathop{\kern 0.1em\vrule width 0.5em height 0.69678ex depth -0.60387ex
                \kern -0.6em \intop}\nolimits_{#1}}}

\theoremstyle{plain}
\newtheorem{theorem}{Theorem}

\newtheorem{example}[theorem]{Example}

\newtheorem{lemma}[theorem]{Lemma}

\numberwithin{theorem}{section}
 \newtheorem{corollary}[theorem]{Corollary}
\newtheorem{remark}[theorem]{Remark}

\theoremstyle{definition}
\newtheorem{dfn}[theorem]{Definition}
\newtheorem*{dfn*}{Definition}

\newcommand{\ud}[0]{\,\mathrm{d}}

\numberwithin{equation}{section}

\begin{document}

\title[Characterizations for arbitrary Békollé-Bonami weights]{Characterizations for arbitrary Békollé-Bonami weights}

\date{\today}

 \author[Carlos Mudarra]{Carlos Mudarra}
\address[C.M.]{Department of Mathematical Sciences, Norwegian University of Science and Technology, 7941 Trondheim, Norway}
\email{carlos.mudarra@ntnu.no}

 \author[Karl-Mikael Perfekt]{Karl-Mikael Perfekt}
\address[K.-M.P.]{Department of Mathematical Sciences, Norwegian University of Science and Technology, 7941 Trondheim, Norway}
\email{karl-mikael.perfekt@ntnu.no}

\makeatletter
\@namedef{subjclassname@2020}{{\mdseries 2020} Mathematics Subject Classification}
\makeatother




\begin{abstract}
We precisely characterize the relationships between the reverse H\"older inequality, the Fujii-Wilson condition, the Békollé-Bonami $\Bp$ condition, the $\Binfty$ condition, and the reverse Jensen inequality, for arbitrary weights in the unit disc. This is achieved by introducing new side conditions that turn out to be necessary and sufficient. The side conditions are simple and testable, and can be interpreted as integral versions of the much stronger condition of bounded hyperbolic oscillation, which has been considered earlier in the literature. 
\end{abstract}

\maketitle

\section{Introduction and Main results}
The Békollé-Bonami class $\Bp(\mathbb{D})$ was introduced in \cite{B81, BB78} in order to classify the weights $w$ in the unit disc $\mathbb{D}$ such that the Bergman projection, as well as the associated maximal function \cite{PR13}, is bounded on $L^p(\mathbb{D}, w)$, $1 < p < \infty$. The Békollé-Bonami condition is similar in form to the Muckenhoupt $A_p$ condition \cite{M72}; one simply replaces the basis of all balls with the much sparser basis of all Carleson squares in the definition. 

The limiting class $A_\infty = \cup_{p > 1} A_p$ plays a major role in the classical weighted theory of harmonic analysis. Accordingly, there are a large number of equivalent characterizations of $A_\infty$ weights. For example a weight $w$ belongs to $A_\infty$ if and only if any of the following conditions apply: $w$ satisfies a reverse H\"older inequality \cite{CN95, HP13, HPR12}, $w$ satisfies the Fujii--Wilson condition \cite{F77,W87}, $w$ satisfies a condition of quantitative absolute continuity \cite{FM74,M73}, or if $w$ satisfies reverse Jensen exponential inequality \cite{H84}. For a detailed exposition on classical results concerning $A_p$ weights, we refer the reader to \cite[Chapter IV]{GCRF85} and \cite[Chapter 7]{G14}.

In the setting of Békollé-Bonami weights, the equivalence between any two of these conditions turns out to be false. The recent preprint \cite{C25} demonstrates this through many illustrative examples. We refer also to the influential paper \cite{DMFO16} where variants of the $A_\infty$-condition are studied for general bases (in particular, for the basis of Carleson squares). The primary obstacle is that pointwise estimates of the form
\begin{equation} \label{eq:pointwise}
w \leq \textrm{const} \cdot \mathrm{M}w
\end{equation}
fail for maximal functions $\mathrm{M}$ formed from sparse bases of sets.

 A natural subclass of weights $w$ in which the equivalences between the various concepts of $\Binfty$ are restored, was identified relatively recently \cite{APR19}. Namely, the class of weights $w$ which are constant in top-halves of Carleson squares, that is, weights of bounded hyperbolic oscillation in the unit disc. Note that the pointwise estimate \eqref{eq:pointwise} holds under this condition. Further studies of Békollé-Bonami for this class and other subclasses of weights can be found in \cite{B04, DLP25, C25, LN23}. 

The purpose of this article is to characterize the pairwise equivalences between various notions of $\Binfty$ in terms of necessary and sufficient side conditions. We begin by defining these key conditions, which for Carleson squares $Q$ are formulated in terms of the dyadic maximal function $\mathrm{M}_Q$,
\[
\mathrm{M}_Qw(x) = \sup_{\substack{Q' \ni x, \, Q' \subset Q \\ Q'\textrm{ dyadic}}} \vint_{Q'} w, \qquad x \in Q,
\]
and the dyadic minimal function 
$\mathrm{m}_Q$,
\[
\mathrm{m}_Qw(x) = \inf_{\substack{Q' \ni x, \, Q' \subset Q \\ Q'\textrm{ dyadic}}} \vint_{Q'} w\, \qquad x \in Q.
\]
More precise definitions of $\mathrm{M}_Q$ and $\mathrm{m}_Q$ are given in Section~\ref{sec:prelim}.

\begin{dfn*} \label{def:Lpandmeasurecond}
Let $w$ be a weight in $\D.$

We say that $w$ satisfies the condition \eqref{eq:Lpconditionmaximal}, when there exist $p_0>1$ and $C>0$ so that
\begin{equation}\label{eq:Lpconditionmaximal}
\int_{Q} w^p \leq C \int_Q (\mathrm{M}_Qw)^p \tag{$\mathrm{M} L^p$}
\end{equation}
for all Carleson boxes $Q$ and all $1<p \leq p_0$.

We say that $w$ satisfies condition \eqref{eq:conditionforFWiffBinfty} if there exist constants $A>1$ and $b\in (0,1)$ so that 
\begin{equation}\label{eq:conditionforFWiffBinfty}
w \left( \lbrace x\in Q\, : \, w(x) \geq A \cdot \mathrm{M}_Q w(x) \rbrace \right) \leq b \cdot w(Q), \tag{$\mathrm{M} \ud w$}
\end{equation}
for every Carleson square $Q.$    

We say that $w$ satisfies the condition \eqref{eq:Lpconditionminimal} when there exist $p_0>1$ and $C>0$ so that
\begin{equation}\label{eq:Lpconditionminimal} 
\frac{1}{w} \in L^{1/(p_0-1)} (\D) \quad \text{and} \quad  \int_Q \left(\frac{1}{w}\right)^{\frac{1}{p-1}} \leq C \int_Q \left(\frac{1}{\mathrm{m}_Qw} \right)^{\frac{1}{p-1}} \tag{$\mathrm{m} L^p$}
\end{equation}
for all Carleson boxes $Q$ and all $p \geq p_0.$

Finally, we say that $w$ satisfies condition \eqref{eq:logconditionminimal} if $\int_{\D} \log w > -\infty$ and there exists $C>0$ so that
\begin{equation}\label{eq:logconditionminimal}
\int_Q \log (  \mathrm{m}_Q(w))  \leq  \int_Q  \log( C w ) \tag{$\mathrm{m}\log$}
\end{equation}
for all Carleson boxes $Q.$ 
\end{dfn*}

Note that condition \eqref{eq:Lpconditionminimal} implies \eqref{eq:logconditionminimal}, by a standard limiting argument. Note also that if the maximal and minimal operators were formed with respect to a basis of sets for which the Lebesgue differentiation theorem holds, all conditions in the definition would be trivially satisfied. In our setting, weights of bounded hyperbolic oscillation also satisfy the conditions in an obvious manner, since the pointwise estimates 
$$
C^{-1} \mathrm{m}_Q w (x) \leq w(x) \leq C \mathrm{M}_Q w(x), \qquad w \in Q,
$$
hold for such weights as well. In general, such estimates are far from holding true in the basis of Carleson squares, even for $\Bp$ weights with the reverse H\"older inequality, cf. Example~\ref{ex:BpandRHInotBHO}.

We are ready to state our main results. The precise definitions of the various $\Binfty$-type conditions are given in Section~\ref{sec:prelim}. We only note here that we have chosen $\Binfty(\D)$ itself to mean the weakest form of the quantitative absolute continuity condition, and for $\Blog(\D)$ to denote the class of weights which satisfy a reverse Jensen inequality with respect to the exponential.

\begin{theorem}\label{thm:maintheoremBinftyiffFW}
The following are equivalent.
\begin{itemize}
    \item[(a)] $w$ satisfies the Fujii-Wilson condition \eqref{eq:definitionofFujiiWilson} and condition \eqref{eq:conditionforFWiffBinfty}.
    \item[(b)] $w$ belongs to $\Binfty(\D)$. 
\end{itemize}
\end{theorem}

\begin{theorem}\label{thm:maintheoremBinftyiffRHI}
The following are equivalent.
\begin{itemize}
    \item[(a)] $w$ satisfies the Fujii-Wilson condition \eqref{eq:definitionofFujiiWilson} and condition \eqref{eq:Lpconditionmaximal}.
    \item[(b)] $w$ satisfies a reverse H\"older inequality \eqref{eq:definitionofRHI}.
\end{itemize}
\end{theorem}

In Theorem~\ref{thm:maintheoremBinftyiffRHI} it is actually enough to verify the inequality of \eqref{eq:Lpconditionmaximal} for one sufficiently small number $p$ (depending on the Fujii-Wilson constant), cf. Remark \ref{rem:verificationofLpestimateinmaintheoremRHI}.

Our other main theorems are as follows.
\begin{theorem}\label{thm:maintheoremBp}
The following are equivalent.
\begin{itemize}
    \item[(a)] $w$ is a $\Bq(\D)$ weight for some $1\leq q < \infty.$
    \item[(b)] $w$ belongs to $\Binfty(\D)$ and satisfies condition \eqref{eq:Lpconditionminimal}.
\end{itemize}
\end{theorem}

\begin{theorem}\label{thm:maintheoremBlog}
The following are equivalent.
\begin{itemize}
    \item[(a)] $w$ is a $\Blog(\D)$ weight.  
    \item[(b)] $w$ belongs to $\Binfty(\D)$ and satisfies condition \eqref{eq:logconditionminimal}.
\end{itemize}
\end{theorem}

Combining Theorems \ref{thm:maintheoremBinftyiffRHI} and \ref{thm:maintheoremBp} with the fact that every $\Bq$ weight satisfies the Fujii-Wilson condition and that the reverse H\"older inequality implies the $\Binfty$ property, we also obtain the following result. 
 
\begin{corollary}\label{cor:RHIiffBq}
The following are equivalent.
\begin{itemize}
    \item[(a)] $w$ satisfies a reverse H\"older inequality \eqref{eq:definitionofRHI} and property \eqref{eq:Lpconditionminimal}. 
    \item[(b)] $w \in \Bq(\D)$ for some $1 \leq q < \infty$ and $w$ satisfies property \eqref{eq:Lpconditionmaximal}. 
\end{itemize}
\end{corollary}

The paper is organized as follows. In Section \ref{sec:prelim} we recall some elementary definitions and various classes of weights and maximal and minimal operators that have been previously considered in the literature. Section \ref{sec:preparatorylemmas} contains various lemmata adapted to the setting of the Carleson square basis, providing alternative arguments and slight improvements in some cases.  In Section \ref{sec:proofsmainresults} we prove Theorems \ref{thm:maintheoremBinftyiffFW}--\ref{thm:maintheoremBlog}. In Section~\ref{sec:examples} we also characterize the logarithms of $\Bp$-weights in terms of bounded mean oscillation. Additionally, we provide two examples of weights that showcase the optimality of our theorems.

\textbf{Notation.} As usual, for any two nonnegative quantities $A$ and $B$ depending on a number of parameters, we write that $A\lesssim B$ or $B \gtrsim A$ if there exists a constant $C\in \left(0,\infty\right)$ such that $A\leq CB$. Furthermore, we write $A\approx B$ whenever there exist constants $C_1, C_2\in \left(0,\infty\right)$ such that $C_1A\leq B \leq C_2A$. This notation is used where the exact magnitude of the constants is not of interest. We indicate the dependence on a variable $x$ of the implied constant with a subscript, $A \lesssim_{x} B$,

\section{Preliminaries and classical conditions} \label{sec:prelim}

Throughout this paper, we consider the unit disc $\D$ equipped with the two-dimensional Lebesgue measure. If $E \subset \D$, the Lebesgue measure of $E$ is denoted by $|E|$. If $|E|>0$ and $f$ is a nonnegative measurable function, we will denote
$$
\vint_E f = \frac{1}{|E|} \int_E f.
$$
Also, in the unit circle $\T$, we consider the normalized Lebesgue measure, so that $|\T|=1$ and for every arc $I \subset \T$, the measure $|I|$ of $I$ is simply the normalized length of $I$ as an arc.

By a \textit{weight} in $\D$, we understand a function $w\in L^1(\D)$ with $0<w(x)<\infty$ for almost every $x\in \D.$ For a measurable set $E \subset \D$, we will frequently use the notation
$$
w(E) : =\int_E w(x) \ud x, \quad w_E:= \vint_E w = \frac{1}{|E|} \int_E w(x)\ud x,
$$
where in the latter definition we naturally assume that $|E|>0.$

The basis of sets over which we will be modeling our conditions for weights is the collection of Carleson squares of the disc. Namely, if $I \subset \T$ in an arc, which can be open, closed, or half-open, the \textit{Carleson box $Q_I$ associated with $I$} is the set
$$
Q_I: = \lbrace z\in \D \, : \, z/|z| \in I \: \text{ and } \: 1-|I| < |z| <1 \rbrace.
$$
A computation shows that
\begin{equation}\label{eq:proportionmeasureannuli}
|Q_I| \leq 4 \left |  Q_{\frac{1}{2}I} \right |, \quad \text{and} \quad \left | Q_I \setminus Q_{(1-\varepsilon) I} \right | \leq 2 \varepsilon |Q_I|, \quad \varepsilon \in (0,1), \quad I \: \text{ arc of } \: \T.
\end{equation}

The \textit{top-half} $T_I$ of the square $Q_I$ is the set
$$
T_{Q_I}:=T_I: =\lbrace z\in \D \, : \, z/|z| \in I \: \text{ and } \: 1-|I| < |z| <1 - |I|/2 \rbrace.
$$
Now, if we fix a half-open interval (or arc) $I$ of $\T,$ we may consider the collection $\mathcal{D}(I)$ of all the right-open \textit{dyadic subarcs of $I$}, which are obtained by recursively bisecting the interval $I$. We also write
$$
\mathcal{D}(Q_I)= \lbrace Q_J \, : \, J \in \mathcal{D}(I) \rbrace.
$$
Observe that if $x\in Q=Q_{I}$, then there exists a unique $J\in \mathcal{D}(I)$ with $x\in T_J.$

In some of our proofs, it will be useful to consider maximal operators with respect to more general measures. In connection with this, let us define doubling measures and doubling weights with respect to Carleson boxes. 

\begin{dfn}[Doubling measures and weights]\label{def:doublingweight}
We say that a measure $\mu$ on $\D$ is doubling when there is some $C>0$ so that 
$$
\mu(Q_I) \leq C \mu  \left( Q_{\frac{1}{2}I}  \right)
$$
for all arcs $I \subset \T.$ Here $\frac{1}{2}I$ denotes that arc of $\T$ with the same center as $I$ and length $|I|/2.$ Note that then there is another constant $C$ so that
\begin{equation}\label{eq:doubmeasureforchildren}
\mu(Q_I) \leq C \mu (Q_{I'}),
\end{equation}
for any arc $I$ and $I' \in \mathcal{D}(I)$ with $|I'| = |I|/2.$

In addition, we say that a weight $w$ in $\D$ is doubling when the corresponding weighted measure $w(x) \ud x$ is doubling on $\D,$ that is, there exists $C>0$ so that
$$
  w(Q_I) \leq C w \left( Q_{\frac{1}{2}I} \right) 
$$
for all arcs $I$ of $\T.$ 
\end{dfn}

In this paper, we will deal precisely with the following maximal operators over Carleson boxes.

\begin{dfn}[Maximal Operators]\label{def:maximaloperator}
Given a function $f \in L^1_{\mathrm{loc}}(\D)$, the maximal function $\M f$ of $f$ is the measurable function
$$
\mathrm{M}f (x):= \sup_{I \subset \T \text{ arc } } \Chi_{Q_I}(x) \cdot \vint_{Q_I} |f|, \qquad x\in \D.
$$
Also, for a measure $\mu$ in $\D$, an $f\in L^1_{\mathrm{loc}}(\D, \mu)$, and a Carleson box $Q_I$, the \textit{dyadic} $\mu$-\textit{maximal function with respect to $Q_I$} is the function
$$
\mathrm{M}^\mu_{Q_{I}} f (x):= \sup_{J \in \mathcal{D}(I)} \Chi_{Q_I}(x) \cdot \vint_{Q_J} |f| \ud \mu, \qquad x\in Q_{I}.
$$
When the underlying measure is the Lebesgue measure, we simply write $\mathrm{M}_{Q_I} f.$ Moreover, when no confusion is possible, we will use $\M^\mu_Q$ or $\M_Q$ in place of $\M^\mu_{Q_I}$ or $\M_{Q_I}.$
\end{dfn}

Now, we consider the corresponding dyadic minimal function for Carleson squares. This will allow us to formulate some of our characterizations in a very neat manner. In the setting of cubes in $\R^n$, the non-dyadic minimal function was introduced by D. Cruz-Uribe and C.J. Neugebauer, establishing the corresponding weighted $L^p$-inequalities; see \cite[Theorem 3.1]{CN95}.  

\begin{dfn}[Dyadic Minimal Operator]\label{def:minimaloperators}
For arcs $I$, one can also define the \textit{dyadic minimal} operator:
$$\mathrm{m}_{Q_{I}} f (x):= \inf_{ J \in \mathcal{D}(I)} \Chi_{Q_J}(x) \cdot \vint_{Q_J} |f|, \qquad x\in Q_{I}.
$$
\end{dfn}

\begin{remark}\label{rem:finitenessminimalmaximal}
{\em
If $w$ is a weight in $\D$ and $Q$ is a Carleson square, then
$$
0< \mathrm{m}_Q w (x), \,   \mathrm{M}_Q w (x) < \infty, \quad \text{for all} \quad x\in Q.
$$
Indeed, if $x\in Q=Q_{I_0},$ there exists a unique $I \in \mathcal{D}(I_0)$ so that $x\in T_I.$ By the properties of the dyadic intervals, there are only finitely many $J \in \mathcal{D}(I_0)$ with $x\in Q_J.$ Therefore 
$$
\mathrm{M}_Q w(x) = \max_{J \in \mathcal{D}(I_0),\, x\in Q_J } \vint_{Q_J} w < \infty
$$
as a maximum of finitely many finite terms. Similarly, using the fact that $\int_E w >0$ for all $E \subset \D$ with $|E| >0,$ one shows that $\mathrm{m}_Q w (x)>0.$
}
\end{remark}

\begin{dfn}[The Fujii-Wilson property and reverse H\"older inequality]\label{def:FWandRHI}
A weight $w$ on $\D$ satisfies the Fujii-Wilson condition when
\begin{equation}\label{eq:definitionofFujiiWilson}
[w]_{\FW}:=\sup \left\lbrace \frac{\int_{Q_I} \mathrm{M}(w \cdot \Chi_{Q_I})}{\int_{Q_I} w} \: : \: I \text{ arc of } \T \right\rbrace < \infty. \tag{$\FW$}    
\end{equation}
On the other hand, we say that $w$ satisfies a reverse H\"older inequality when there is $p>1$ so that
\begin{equation}\label{eq:definitionofRHI}
[w]_{\RHI_p}:=\sup \left\lbrace \frac{ \vint_{Q_I} w^p  }{ \left( \vint_{Q_I} w \right)^p} \: : \: I \text{ arc of } \T \right\rbrace < \infty. \tag{$\mathrm{RHI}$}
\end{equation}
\end{dfn}

It is important to mention that in \cite{APR19}, the class of weights with the Fujii-Wilson property is denoted by $\Binfty.$ Our definition of $\Binfty$ and the rest of the $\Bp$ classes is as follows.

\begin{dfn}[$\Binfty$ and Békollé-Bonami weights]\label{def:BpandBinfty}
Let $w$ be a weight in $\D.$

We say that $w$ is a $\Binfty(\D)$-weight if there are constants $\alpha, \beta \in (0,1)$ so that for every arc $I \subset \T$ and every measurable set $E \subset Q_I,$ one has
 \begin{equation}\label{eq:defBifnty}
|E| \leq \alpha |Q_I| \implies w(E) \leq \beta w(Q_I). \tag{$\Binfty$}
\end{equation}
W say that $w$ is a $\Bp(\D)$-weight, $1< p < \infty,$ if
 \begin{equation}\label{eq:defBp}
    [w]_p:=\sup_{I \text{ arc of } \T} \left( \vint_{Q_I} w \right)  \left( \vint_{Q_I} w^{\frac{1}{1-p}} \right)^{p-1}  < \infty. \tag{$\mathrm{B}_p$}
\end{equation}
We say that $w$ is a $\Blog(\D)$-weight if $\int_{\D} \log w > -\infty$ and
\begin{equation}\label{eq:defBlog}
    [w]_{\log} :=\sup_{I \text{ arc of } \T} \left( \vint_{Q_I} w \right) \exp \left( - \vint_{Q_I} \log w \right)  < \infty, \tag{$\mathrm{B}_{\log}$} 
\end{equation}
where we understand that $\exp(-\infty)=0.$
\end{dfn}

Regarding the classes defined in Definition \ref{def:BpandBinfty}, the following information is well known.

\begin{itemize}
    \item The condition \eqref{eq:defBifnty} is analogous to the $\mathrm{A}_\infty(\R^n)$-condition for cubes considered by C. Fefferman and B. Muckenhoupt \cite{FM74}. 
    \item In the limiting case $p \to 1,$ we get the class $\mathrm{B}_1(\D)$, consisting of those $w$ with
    $$
    \sup_{I \text{ arc of } \T }  \left( \vint_{Q_I} w \right) \cdot \esssup_{Q_I} \left( \frac{1}{w} \right) < \infty.
    $$
One therefore has, by virtue of the Jensen Inequality, that $\mathrm{B}_1(\D) \subset \Bp(\D) \subset \Bq(\D)$ for $1 \leq p \leq q.$
\item Furthermore, Jensen's inequality for the convex function $t \mapsto \exp(-t)$ shows that $\Bp(\D) \subset \Blog(\D) $ for all $1 \leq p < \infty.$ 

\end{itemize}

Less elementary, but still well known, is the inclusion $\Blog(\D) \subset \Binfty(\D).$ 

\begin{lemma}\label{lem:BpimpliesBinfty}
If $w\in \Blog (\D)$ for $1\leq p< \infty,$ then $w\in \Binfty(\D)$ as well.
\end{lemma}
\begin{proof}
A very direct proof of this result is in \cite[Theorem 1]{H84} in the setting of cubes in $\R^n$, but its proof is valid for general measure spaces and bases. For a different proof, proving other interesting conditions on the way, see \cite[Theorem 4.1]{DMFO16}.
\end{proof}

We will make use of the following equivalent form of $\mathrm{B}_\infty.$

\begin{lemma}\label{lem:equivalencBinftysuperlevelsets}
A weight $w\in \mathrm{B}_\infty(\D)$ if and only if there are constants $\alpha, \beta \in (0,1)$ such that
$$
w \left( \lbrace x\in Q \, : \, w(x) \geq \frac{1}{\alpha} w_Q \rbrace \right) \leq \beta w(Q),
$$
for all Carleson boxes $Q$.
\end{lemma}
\begin{proof}
This is well known and a proof can be found, for instance, in \cite[Theorem 3.1]{DMFO16} for general measure spaces and bases.    
\end{proof}

Concerning the connection between $\mathrm{B}_p$ weights and the reverse H\"older inequality \eqref{eq:definitionofRHI}, perhaps the only easy implication is given by the following lemma, based on a direct application of the classical H\"older inequality. We have chosen to include it for the sake of completeness. 

\begin{lemma}\label{lem:RHIimpliesBinftyanddoub}
If $w$ satisfies a reverse H\"older inequality, then $w\in \Binfty(\D)$.  
\end{lemma}
\begin{proof}
Let $p>1$ the exponent for which $w$ satisfies the \eqref{eq:definitionofRHI}. Also, let $I \subset \T$ be an arc, and $E \subset Q_I$ a measurable set. Applying first the classical H\"older inequality and then \eqref{eq:definitionofRHI}, we get
$$
\frac{w(E)}{|Q_I|} = \vint_{Q_I} w \cdot \Chi_{ E} \leq \left( \vint_{Q_I} w^p \right)^{1/p} \left( \frac{| E|}{|Q_I|} \right)^{1/p'} \leq [w]_{\mathrm{RHI}_p}^{1/p} \left( \vint_{Q_I} w\right)  \left( \frac{|  E|}{|Q_I|} \right)^{1/p'},
$$
which, after simplifying leads to
$$
\frac{w(E)}{w(Q_I)} \leq [w]_{\mathrm{RHI}_p}^{1/p}  \left( \frac{|  E|}{|Q_I|} \right)^{1/p'}.
$$
It is then immediately seen that $w$ satisfies the \eqref{eq:defBifnty} condition. 
\end{proof}

\section{Preparatory lemmas}\label{sec:preparatorylemmas}

\subsection{Calderón-Zygmund decompositions}

We will make frequent use of Calderón-Zygmund decompositions for the dyadic maximal function with respect to Carleson squares.  In this setting, the Calderón-Zygmund squares $\lbrace Q_j^\lambda \rbrace$ of a function $f \geq 0$ over a cube $Q$ provide an exact partition of the set $Q \cap \lbrace \mathrm{M}_Q f > \lambda \rbrace.$ Unlike the classical setting, we cannot necessarily relate this set to the distributional set $Q \cap \lbrace f > \lambda \rbrace,$ due to the falsity of estimates of the type $f \lesssim \mathrm{M} f$.

We begin with an elementary observation. 
\begin{lemma}\label{lem:maximalsubcollections}
Let $I_0 \subset \T,$ and $\mathcal{F}  \subset \mathcal{D}(I_0)$ a subcollection. Let $\mathcal{F}^*$ be the collection of maximal arcs of $\mathcal{F}:$ those $I \in \mathcal{F}$ that are not strictly contained in any larger $I'\in \mathcal{F}.$ Then
\begin{itemize}
\item Every $I\in \mathcal{F}$ is contained in a unique $I^* \in \mathcal{F}^*.$ And so, every $Q_I$, $I\in \mathcal{F},$ is contained in a unique $Q_{I^*},$ $I^* \in \mathcal{F}^*.$
\item $\bigcup_{I\in \mathcal{F}} I = \bigcup_{I^*\in \mathcal{F}^*} I^*$ and $\bigcup_{I\in \mathcal{F}} Q_I = \bigcup_{I^*\in \mathcal{F}^*} Q_{I^*}.$
\item The intervals of $\mathcal{F}^*$ are mutually disjoint. Thus, $\lbrace Q_{I^*} \rbrace_{I^*\in \mathcal{F}^*}$ are mutually disjoint. 
\end{itemize}  

\end{lemma}
\begin{proof}
It follows from the fact that if $I,I'\in \mathcal{D}(I_0),$ then $I \cap I' \in \lbrace \emptyset, I, I' \rbrace$, and that any chain $I_1 \subset I_2 \subset \cdots  $ has $I_0$ as an upper bound, thus stopping after finitely many steps. 
\end{proof}
 
The previous lemma enables us to obtain the Calderón-Zygmund decomposition of a function, with almost all the properties that one has in the setting of cubes in $\R^n$. We include the well known proof for completeness.

\begin{lemma}\label{lem:CZdecomposition}
Let $\mu$ be a doubling measure in $\D$, let $f: \D \to [0, \infty]$ be a nonnegative $L^1(\D, \mu)$ function, let $I \subset \T$ be an arc, denote $Q=Q_I$, and assume $ \vint_{Q} f \ud \mu \leq \lambda >0.$ Then, there is a countable collection of arcs $\lbrace I_j \rbrace_j \subset \mathcal{D}(I)$ such that
\begin{itemize}
\item $\lbrace I_j \rbrace_j$ are pairwise disjoint.
\item For every $j,$
$$
\lambda < \vint_{Q_{I_j}} f \ud \mu \leq C(\mu) \lambda.
$$
\item Any $J \in \mathcal{D}(I)$ that properly contains some $I_j$ satisfies $\vint_{Q_J} f \ud \mu \leq \lambda.$
\item $\lbrace x\in Q \, : \,  \mathrm{M}_{Q}^\mu f (x) > \lambda \rbrace = \bigcup_j Q_{I_j}.$ Consequently,
\begin{equation}\label{eq:estimateCZlevelsets}
\lambda \cdot \mu \left( Q \cap \lbrace \mathrm{M}_Q^\mu f > \lambda  \rbrace \right) \leq  \int_{ Q \cap \lbrace \mathrm{M}_Q^\mu f > \lambda \rbrace} f \leq  C(\mu) \cdot \lambda \cdot \mu \left( Q \cap \lbrace \mathrm{M}_Q^\mu f > \lambda \rbrace \right).    
\end{equation}

\end{itemize}
In the case where $\mu$ is the Lebesgue measure on $\D,$ one can take $C(\mu) = 4.$ 
\end{lemma}
\begin{proof}
Define $\mathcal{F}:= \lbrace J \in \mathcal{D}(I) \, : \, \vint_{Q_J} f \ud \mu > \lambda \rbrace,$ and assume that $\mathcal{F} \neq \emptyset,$ as otherwise the claim holds vacuously. Consider the maximal arcs $\mathcal{F}^*$ of $\mathcal{F}$ as in Lemma \ref{lem:maximalsubcollections}, and set $\lbrace I_j \rbrace_j:= \mathcal{F}^*$. 
The first property and the first inequality of the second follow immediately. For the upper bound for $\vint_{Q_{I_j}} f \ud \mu,$ note that $I_j$ is strictly contained in $I$ (by the assumption $\vint_Q f \ud \mu \leq \lambda$), and so there is a dyadic parent $J \in \mathcal{D}(I)$ of $I_j,$ which satisfies $\mu(Q_{J}) \leq C(\mu)  \mu (Q_{I_j})$ by the doubling property \eqref{eq:doubmeasureforchildren}. 
Note that \eqref{eq:proportionmeasureannuli} gives $\mu(Q_J)\leq 4 \mu(Q_{I_j})$ when $\mu$ is the Lebesgue measure. 
Moreover, by the maximality of $I_j,$ we must have $\vint_{Q_J} f \ud \mu \leq \lambda,$ and so
$$
\vint_{Q_{I_j}} f \ud \mu \leq \frac{1}{\mu(Q_{I_j})} \int_{Q_J} f \ud \mu \leq  \lambda \frac{\mu(Q_J)}{\mu(Q_{I_j})} \leq C(\mu)  \lambda.
$$
The third claim is a consequence of the maximality of the $I_j$'s. For the last property, note that if $\mathrm{M}_{Q}^\mu f(x) > \lambda$ with $x\in Q_{I}$, then there is $J \in \mathcal{D}(I)$ so that $Q_J$ contains $x$ and $\vint_{Q_J} f \ud \mu > \lambda.$ Thus $Q_J \in \mathcal{F},$ and so it is contained in some $Q_{I_j}$ with $I_j \in \mathcal{F}^*,$ whence $x\in Q_{I_j}.$ The reverse inclusion is a consequence of the second claim. 
\end{proof}

\subsection{The maximal and minimal operators for doubling weights}
\begin{lemma}\label{lem:equivalentmaximalfunctions}
Let $w$ be a weight in $\D$, and $Q_0 := Q_{I_0}$ a Carleson square. Then
$$ 
 \mathrm{M}(w \cdot \Chi_{Q_0}) (x) = \sup_{I  \subset I_0} \Chi_{Q_I}(x) \cdot \vint_{Q_I} w, \quad \quad x\in Q_0.
$$
Furthermore, if $w$ is doubling in $\D,$ then 
$$
 \mathrm{M}(w \cdot \Chi_{Q_0}) (x) \approx  \mathrm{M}_{Q_0} w(x), \quad \quad x\in Q_0.
$$
and also
\begin{equation}\label{eq:comparisondyadicminimalandcontainedminimal}
      \mathrm{m}_{Q_0} w(x) \approx  \inf_{I \subset I_0} \Chi_{Q_I}(x) \cdot \vint_{Q_I} w, \quad x\in Q_0,
\end{equation}
with implied constants depending only on $w$.
\end{lemma}
\begin{proof}
Let $x\in Q_0$ and $Q=Q_I$ be a Carleson box associated with $I \subset \T$ and $x\in Q.$ In the case where $|I| \geq |I_0|$, we simply estimate 
$$
 \vint_{Q} (w \cdot \Chi_{Q_0}) \leq  \vint_{Q_0} w.
$$
Now, if $|I| < |I_0|$, since $x\in Q_I \cap Q_{I_0}$, it is clear, by shifting $I$ inside $I_0$, that there exists an interval $I' \subset I_0$ containing $I \cap I_0$ with $|I'|= |I|$. Therefore $Q_{I'} \subset Q_0$, $x\in Q_0 \cap Q_I \subset Q_{I'} $, and
$$
\frac{1}{|Q_{I}|} \int_{Q_{I}} w \cdot \Chi_{Q_0} \leq \frac{1}{|Q_{I'}|} \int_{Q_{I'}} w \cdot \Chi_{Q_0} =\frac{1}{|Q_{I'}|} \int_{Q_{I'}} w.
$$
This proves the first equality.

Assume now that $w$ is doubling. Let $x\in Q_0$, and $I \subset I_0$ an arc with $x\in Q_I \subset Q_0.$ Let $k\in \N \cup \lbrace 0 \rbrace$ so that
$$
2^{-k-1} |I_0| \leq |I| \leq 2^{-k} |I_0|.
$$
Denote $\mathcal{D}_j(I_0):= \lbrace J \in \mathcal{D}(I_0) \, : \, |J| = 2^{-j} |I_0| \rbrace$ for every $j \in \N \cup \lbrace 0 \rbrace$. Note that $I$ intersects at most $2$ intervals of the family $\mathcal{D}_k(I_0),$ as otherwise $I$ would intersect at least $3$ intervals of $\mathcal{D}_k(I_0)$, thus containing some $I' \in \mathcal{D}_k(I_0)$ with $|I| > |I'| = 2^{-k} |I_0|,$ a contradiction. Let $I_1, I_2 \in \mathcal{D}_k(I_0)$ those for which $I_1\cap I, I_2 \cap I \neq \emptyset.$ Since $\mathcal{D}_k(I_0)$ defines a partition of $I_0$, we have that $I \subset I_1 \cup I_2$. Also, note that 
$$
\frac{1}{2} |I_j| \leq |I| \leq |I_j|, \quad j=1,2.
$$
This implies that $Q_I \subset Q_{I_1} \cup Q_{I_2}$ and, according to \eqref{eq:proportionmeasureannuli}, 
$$
\frac{1}{4} |Q_{I_j}| \leq |Q_I| \leq |Q_{I_j}|, \quad j=1,2.
$$
Assume that $Q_{I_1}$ is the box containing $x$. If $J: = I_1 \cup I_2$ (not necessarily a dyadic sub-interval of $I_0$), one has that $2 |I_1 | = |J| $, and by the doubling property of $w$:
$$
w(Q_I) \leq w(Q_J) \lesssim_w w(Q_{I_1}). 
$$
Therefore, we have the estimate
$$
 \frac{1}{|Q_{I}|} \int_{Q_{I}} w \lesssim_w  \frac{1}{|Q_{I_1}|} \int_{Q_{I_1}} w \leq \mathrm{M}_{Q_0} w (x).
$$
But also, since $|I| \geq |I_1|/2, |I_2|/2$ we have that $|I| \geq |J|/4$. Since $I \subset J,$ the doubling condition of $w$ tells us that $w(Q_I) \gtrsim_w w(Q_J)$ and therefore
\begin{equation*}
 \frac{1}{|Q_{I}|} \int_{Q_{I}} w \approx  \frac{1}{|Q_{I_1}|} \int_{Q_{I}} w \gtrsim_w  \frac{1}{|Q_{I_1}|} \int_{Q_{J}} w \geq \frac{1}{|Q_{I_1}|} \int_{Q_{I_1}} w \geq \mathrm{m}_{Q_0} w (x). \qedhere
\end{equation*}
\end{proof}

\begin{remark}\label{rem:doublingweightreformulationconditions} \rm If $w$ is a doubling weight, then Lemma \ref{lem:equivalentmaximalfunctions} shows that conditions \eqref{eq:Lpconditionmaximal} and \eqref{eq:conditionforFWiffBinfty} can be formulated with $\M(w \cdot \Chi_Q)$ in place of the dyadic maximal function $\M_Q w$. 

Similarly, for doubling weights $w$, the conditions \eqref{eq:Lpconditionminimal} and \eqref{eq:logconditionminimal} can be reformulated with the minimal function defined in the right-hand side of \eqref{eq:comparisondyadicminimalandcontainedminimal} instead of $\mathrm{m}_Q w$.
\end{remark}

\subsection{The Fujii-Wilson property and the class $\mathrm{B}_\infty(\D)$} Here we establish two important lemmas concerning \eqref{eq:defBifnty} weights and the Fujii-Wilson \eqref{eq:definitionofFujiiWilson} property.

The implication $\eqref{eq:defBifnty} \implies \eqref{eq:definitionofFujiiWilson}$ was proven in \cite[Theorem 6.1]{DMFO16} for the Carleson square basis in the half-plane. There, the authors used an argument based on a discretized version of the usual maximal function. Here we offer an alternative proof based on Calderón-Zygmund decompositions at arbitrarily large scales, making use of the fact that for doubling weights $w$ one can replace $\mathrm{M}(w \cdot \Chi_Q)$ with $\mathrm{M}_Q w$; see Lemma \ref{lem:equivalentmaximalfunctions}.

\begin{lemma}\label{lem:BinftyimpliesFWanddoubling}
Let $w$ be a weight in $\D$ with the \eqref{eq:defBifnty} condition. Then $w$ is doubling and satisfies the Fujii-Wilson property \eqref{eq:definitionofFujiiWilson}.  
\end{lemma}
\begin{proof}
Let $\alpha,\beta \in (0,1)$ be as in the \eqref{eq:defBifnty}-condition.

We will first verify that $w$ is a doubling weight. For each arc $I \subset \T$, and $\varepsilon \in (0,1)$ consider the box $E_\varepsilon= Q_I \setminus Q_{(1-\varepsilon) I}$; where $(1-\varepsilon) I$ has same center as $I$ and length $(1-\varepsilon)|I|$. By \eqref{eq:proportionmeasureannuli}, we get that $|E_\varepsilon| \leq 2 \varepsilon |Q_I|.$ Thus we can take $\varepsilon= \varepsilon(\alpha) >0$ small enough so that $|E_\varepsilon| \leq \alpha |Q_I|,$ and therefore $w(E_\varepsilon) \leq \beta w(Q_I).$ In other words, 
$$
w( Q_{(1-\varepsilon) I} ) \geq (1-\beta) w(Q_I).
$$
Iterating this inequality, we get that 
$$
w\left( Q_{\frac{1}{2} I} \right)  \geq (1-\beta)^{N(\alpha)} w(Q_I),
$$
for a natural number $N(\alpha)$ depending only on $\alpha$.

We now show the Fujii-Wilson property for $w.$ Fix a Carleson box $Q,$ and let $\lambda>1$ be large enough so that $4/\lambda \leq \alpha$. Then, for each $k\in \N \cup \lbrace 0 \rbrace$, let $\lbrace Q_k^j \rbrace_j$ be the dyadic Calderón-Zygmund subcubes of $Q$ of level $w_Q \cdot \lambda^k$ from Lemma \ref{lem:CZdecomposition}. We define the sets
\begin{equation}\label{eq:definitionFkEkj}
F_k^j :  = Q_k^j \setminus \bigcup_{l} Q_{k+1}^l, \quad E_k^j:= Q_k^j \setminus F_k^j.
\end{equation}
Observe that if $l,j$ are indices so that $Q_k^j \cap Q_{k+1}^l \neq \emptyset,$ then, by the properties of the dyadic cubes, we have either $Q_k^j \subset Q_{k+1}^l$ or $Q_{k+1}^l \subset Q_k^j$. But since
$$
\vint_{Q_{k+1}^l} w > w_Q \lambda^{k+1} > w_Q \lambda^k,
$$
the first situation is impossible, by the maximality of $Q_k^j$ as a cube of level $w_Q \lambda^k$. Therefore, $Q_{k+1}^l \subset Q_k^j.$ This observation will be taken into account to estimate the Lebesgue measure of $E_k^j,$ in combination with the properties of the Calderón-Zygmund cubes (see Lemma \ref{lem:CZdecomposition}):
\begin{align*}
|E_k^j| & = | Q_k^j \setminus F_k^j | = \sum_{ \lbrace l \, : \, Q_k^j \cap Q_{k+1}^l \neq \emptyset \rbrace} | Q_k^j \cap Q_{k+1}^l |   \leq   \sum_{ \lbrace l \, : \, Q_k^j \cap Q_{k+1}^l \neq \emptyset \rbrace} |   Q_{k+1}^l | \\[0.5em]
& \leq  \sum_{ \lbrace l \, : \, Q_k^j \cap Q_{k+1}^l \neq \emptyset \rbrace} \frac{w(Q_{k+1}^l)}{w_Q \cdot \lambda^{k+1}} \leq \frac{w(Q_k^j)}{w_Q \cdot \lambda^{k+1}}  \leq \frac{4 w_Q \cdot \lambda^k}{w_Q \cdot \lambda^{k+1}} |Q_k^j| = \frac{4}{\lambda} |Q_k^j| \leq \alpha |Q_k^j|.
\end{align*}
By the \eqref{eq:defBifnty} condition, we get that $w(E_k^j) \leq \beta w(Q_k^j),$ from which
\begin{equation}\label{eq:estimateweightedmeasureFkj}
w(F_k^j) \geq (1-\beta) w(Q_k^j), \quad \text{for all} \quad k,j.    
\end{equation}
It is clear that $Q \cap \lbrace \mathrm{M}_Q w > w_Q \rbrace$ can be written as the disjoint union the sets $\lbrace F_k^j\rbrace_{k,j}$, by \eqref{eq:definitionFkEkj} and the trivial fact that $\mathrm{M}_Q w$ is always larger than $w_Q$. Using \eqref{eq:estimateweightedmeasureFkj}, we have
\begin{align*}
\int_{Q \cap \lbrace \mathrm{M}_Q w> w_Q \rbrace} \mathrm{M}_Q w & = \sum_{k} \sum_j \int_{F_k^j} \mathrm{M}_Q w \leq \sum_{k} \sum_j  w_Q  \lambda^k | F_k^j | \leq \sum_{k} \sum_j \frac{| F_k^j |}{|Q_k^j|} w(Q_k^j) \\[0.5em]
& \leq \frac{1}{1-\beta}  \sum_{k} \sum_j w(F_k^j) = \frac{1}{1-\beta} w(Q) = \frac{1}{1-\beta} \int_Q w,
\end{align*}
whence
$$
\int_Q \mathrm{M}_Q w =\int_{Q \cap \lbrace \mathrm{M}_Q w = w_Q \rbrace} \mathrm{M}_Q w + \int_{Q \cap \lbrace \mathrm{M}_Q > w_Q \rbrace} \mathrm{M}_Q w \leq |Q| w_Q + \frac{1}{1-\beta}\int_Q w = \frac{2-\beta}{1-\beta} \int_Q w.
$$
Now, since we have already proved that $w$ is doubling, Lemma \ref{lem:equivalentmaximalfunctions} can be combined with the above to conclude that
$$
\int_Q \mathrm{M}( w \cdot \Chi_Q) \leq C(w) \int_Q \mathrm{M}_Q w \leq C(w,\beta)\int_Q w,
$$
thus obtaining the Fujii-Wilson inequality.

\end{proof}

The proof of the following lemma is a slight modification of \cite[Lemma 2.2]{HPR12}, adapted to the basis of Carleson squares. In the proof we also prioritize obtaining a better power $p$, at the cost of a worse multiplicative constant, cf. Remark \ref{rem:verificationofLpestimateinmaintheoremRHI}.

\begin{lemma}\label{lemma:selfimprovedyadicmaximal}
Assume that $w$ satisfies the Fujii-Wilson condition \eqref{eq:definitionofFujiiWilson}, and denote $L=[w]_{\FW}$. Then, for every
$$
1< p < \frac{4 L}{4L-1}
$$
one has
$$
\vint_{Q} \left( \mathrm{M}_{Q} w \right)^{p} \leq  \frac{L}{1-4 L/p' } \left( \vint_{Q} w \right)^{p}, 
$$
for every Carleson square $Q$. Here $p'$ is the dual index to $p$.
\end{lemma}
\begin{proof}
Fix a Carleson square $Q=Q_I,$ and for each $\lambda \geq \vint_{Q} w,$ consider the Calderón-Zygmund intervals $\lbrace I_j^\lambda \rbrace_j$ at level $\lambda$ of $w$ as in Lemma \ref{lem:CZdecomposition}, denote the corresponding boxes by $\lbrace Q_j^\lambda \rbrace_j$, $Q_j^\lambda = Q_{I_j^\lambda}$, and also denote $\Omega_\lambda:= Q \cap \lbrace \mathrm{M}_Q w > \lambda \rbrace.$ Then, using the estimate \eqref{eq:definitionofFujiiWilson} for $Q$, we have the bound $\mathrm{M}_Q w (\Omega_\lambda) \leq L w(Q),$ which allows us to write
\begin{align*}
\int_{Q} (\mathrm{M}_Q w )^p \ud x & = \int_0^\infty (p-1) \lambda^{p-2} (\mathrm{M}_{Q} w)( \Omega_\lambda) \ud \lambda \\
& = \int_0^{w_{Q}} (p-1) \lambda^{p-2} (\mathrm{M}_{Q} w)( \Omega_\lambda) \ud \lambda +  \int_{w_{Q}}^\infty (p-1) \lambda^{p-2} (\mathrm{M}_{Q} w)( \Omega_\lambda) \ud \lambda \\
& \leq  L (w_Q)^{p-1} w(Q) + \int_{w_{Q}}^\infty (p-1) \lambda^{p-2} \left( \sum_j \int_{Q_j^\lambda}  (\mathrm{M}_{Q} w)(x) \ud x \right) \ud \lambda .
\end{align*}
 Given $x\in Q_j^\lambda$, if $x\in Q' =Q_J$ for some $J \in D(I)$, then either $Q' \subset Q_j^\lambda$ or $\vint_{Q'} w \leq \lambda$, as a consequence of the maximality of the Calderón-Zygmund cubes; see Lemma \ref{lem:CZdecomposition}. Therefore the supremum defining $\mathrm{M}_Q w (x)$ satisfies 
$$ 
\mathrm{M}_{Q} w (x)  = \sup_{J \in \mathcal{D}(I), \, J \subset I_{j}^\lambda} \Chi_{Q_J}(x) \cdot  \vint_{Q_J} w = \sup_{J\in \mathcal{D}(I_j^\lambda)} \Chi_{Q_J}(x) \cdot \vint_{Q_J} w = \mathrm{M}_{Q_j^\lambda} w (x), \quad x\in Q_j^\lambda. 
$$
This observation says that the last integral coincides with
$$
 \int_{w_{Q}}^\infty (p-1) \lambda^{p-2} \left( \sum_j \int_{Q_j^\lambda}  (\mathrm{M}_{Q_j^\lambda} w)(x) \ud x \right) \ud \lambda.
$$
By \eqref{eq:definitionofFujiiWilson} and the properties from Lemma \ref{lem:CZdecomposition} and, in particular,  \eqref{eq:estimateCZlevelsets}, this term is smaller than
\begin{align*}
 L \int_{w_{Q}}^\infty (p-1)   \lambda^{p-2} & \left( \sum_j  \int_{Q_j^\lambda} w(x) \ud x  \right) \ud \lambda    =  L \int_{w_{Q}}^\infty (p-1) \lambda^{p-2} w \left( Q \cap \lbrace \mathrm{M}_Q w > \lambda \rbrace \right) \ud \lambda \\[0.5em]
& \leq 4 L \int_{w_{Q}}^\infty (p-1) \lambda^{p-1} | Q \cap \lbrace \mathrm{M}_Q w > \lambda \rbrace | \ud \lambda \\[0.5em]
& \leq  4 L \int_{0}^\infty (p-1) \lambda^{p-1} | Q \cap \lbrace \mathrm{M}_Q w > \lambda \rbrace | \ud \lambda \leq \frac{4 L (p-1)}{p} \int_Q (\mathrm{M}_Q w)^{p}.
\end{align*}
Inserting this back into the inequality, averaging the integrals over $Q,$ and using the upper bound for $p$ from the assumption, we deduce the desired inequality. 
\end{proof}

\section{Optimal characterizations: proofs of the main results}\label{sec:proofsmainresults}

In this section, we prove our main results in the following order: Theorem \ref{thm:maintheoremBinftyiffFW}, Theorem \ref{thm:maintheoremBinftyiffRHI}, and then jointly Theorems~\ref{thm:maintheoremBp} and \ref{thm:maintheoremBlog}.

\subsection{Characterizing \eqref{eq:defBifnty} in terms of \eqref{eq:definitionofFujiiWilson}}

In this subsection, we prove Theorem \ref{thm:maintheoremBinftyiffFW}. Let us begin with some observations concerning the condition \eqref{eq:conditionforFWiffBinfty}. 

\begin{remark}
{\em 
Let $w$ be a weight in $\D.$
\begin{enumerate}
    \item Condition \eqref{eq:conditionforFWiffBinfty} for $w$ is equivalent to saying that there are $A>1$ and $b\in (0,1)$ so that
    $$
    w \left( \lbrace x\in Q\, : \, w(x) < A \cdot \mathrm{M}_Q w(x) \rbrace \right) \geq (1-b) \cdot w(Q)
    $$
    for all Carleson boxes $Q.$ In other words, in terms of the weighted $w$-measure, the set of points $Q$ where $w$ is dominated by the dyadic maximal function contains a uniform portion of $Q.$

    \item A natural stronger version of \eqref{eq:conditionforFWiffBinfty} is as follows:
   \begin{equation}\label{eq:conditionFWiffRHIForeverythereexists}
\text{For any  } \:  b\in (0,1) \: \text{ there is } \: A>1 \: \text{ s.t. } \:       w \left( \lbrace x\in Q\, : \, w(x) \geq A \cdot \mathrm{M}_Q w(x) \rbrace \right) \leq b \cdot w(Q), \tag{$\mathrm{M} \ud w^+$}
   \end{equation}
for every Carleson box $Q.$  We will show in Corollary~\ref{cor:Binfty+iffMAb+andFW} that this opens up for a characterization of absolute continuity in a stronger sense than that of \eqref{eq:defBifnty}.

    \item If $w$ is a doubling weight, then \eqref{eq:conditionforFWiffBinfty} and \eqref{eq:conditionFWiffRHIForeverythereexists} can be rewritten with the maximal function $\mathrm{M}(w \cdot \Chi_Q)  $ in place of the dyadic maximal function $\mathrm{M}_Q w$, due to Lemma \ref{lem:equivalentmaximalfunctions}.   

\end{enumerate}

}    
\end{remark}

\begin{lemma}\label{lem:BinftyimpliesMAbcondition}
If $w\in \mathrm{B}_\infty(\D)$, then $w$ satisfies condition \eqref{eq:conditionforFWiffBinfty}.
\end{lemma}
\begin{proof}
By Lemma \ref{lem:equivalencBinftysuperlevelsets}, condition \eqref{eq:defBifnty} implies the existence of constants $ \alpha, \beta \in (0,1)$ with
$$
w \left( \lbrace x\in Q\, : \, w(x) \geq \frac{w_Q}{\alpha} \rbrace \right) \leq \beta w(Q)
$$
for all Carleson boxes $Q.$ The trivial pointwise estimate $w_Q \leq \mathrm{M}_Q w(x)$, $x\in Q,$ thus yields
$$
w \left( \lbrace x\in Q\, : \, w(x) \geq \frac{1}{\alpha} \mathrm{M}_Q w(x) \rbrace \right) \leq  w \left( \lbrace x\in Q\, : \, w(x) \geq \frac{w_Q}{\alpha} \rbrace \right) \leq \beta w(Q),
$$
which is \eqref{eq:conditionforFWiffBinfty} for $A=1/\alpha$ and $b=\beta.$
\end{proof}

\begin{lemma}\label{lem:MAbconditioandFwimplyBinfty}
If $w$ satisfies \eqref{eq:definitionofFujiiWilson} and \eqref{eq:conditionforFWiffBinfty}, then $w\in \mathrm{B}_\infty(\D)$.
\end{lemma}
\begin{proof}
Let $A >1$ and $b\in (0,1)$ be as in \eqref{eq:conditionforFWiffBinfty}, and denote by $L=[w]_{\mathrm{FW}}$ the Fujii-Wilson \eqref{eq:definitionofFujiiWilson} constant of $w.$ Let $A^*>1$ be a large parameter whose value will be specified at the end of the proof, depending on $A,b,$ and $L$. 

Using \eqref{eq:conditionforFWiffBinfty} we can estimate
\begin{align}\label{eq:sublevelsetswMQwlongestimates1}
w & \left( \lbrace x\in Q \, : \, w(x) \geq A^* \cdot w_Q \rbrace \right) \nonumber \\[0.5em]
& \leq w\left( \lbrace x\in Q \, : \,  w(x) \geq A \cdot \mathrm{M}_Q w(x) \rbrace \right)     +  w\left( \lbrace x\in Q \, : \, A^* \cdot  w_Q \leq w(x) < A \cdot \mathrm{M}_Q w (x) \rbrace \right) \nonumber \\[0.5em]
& \leq b \cdot w(Q)     +  w\left( \lbrace x\in Q \, : \, A^* \cdot w_Q \leq w(x) < A \cdot \mathrm{M}_Q w (x) \rbrace\right).
\end{align}
We now focus on estimating the second term of \eqref{eq:sublevelsetswMQwlongestimates1}. We first apply the H\"older inequality for some $1< p < 4L /(4L-1), $ and then Lemma \ref{lemma:selfimprovedyadicmaximal} precisely for that $p$:
\begin{align}\label{eq:sublevelsetswMQwlongestimates2}
 w  &  \left( \lbrace x\in Q \,  : \, A^* \cdot  w_Q \leq w(x) < A \cdot  \mathrm{M}_Q w(x) \rbrace  \right) \nonumber \\[0.5em]
 & \qquad  \qquad = \int_{Q \cap \lbrace A^* \cdot w_Q \leq w \leq A \cdot \mathrm{M}_Q w  \rbrace } w(x)   \ud x \leq \int_{Q \cap \lbrace A^* \cdot  w_Q \leq w \leq A \cdot \mathrm{M}_Q w \rbrace } A \cdot \mathrm{M}_Q w(x)   \ud x \nonumber \\[0.5em]
 & \qquad \qquad \leq \int_{Q \cap \lbrace \mathrm{M}_Q w  \geq (A^*/A) \cdot w_Q  \rbrace } A \cdot \mathrm{M}_Q w(x)   \ud x \nonumber \\[0.5em]
 & \qquad \qquad \leq A \left( \int_Q \left(  \mathrm{M}_Q w \right)^p \right)^{1/p} \left | Q \cap \lbrace  \mathrm{M}_Q w  \geq (A^*/A) \cdot w_Q  \rbrace \right |^{1/p'} \nonumber \\[0.5em]
 & \qquad \qquad \leq  A \left( \frac{L}{1-(4 L)/p' } \right)^{1/p} |Q|^{\frac{1}{p}-1} w(Q) \left | Q \cap \lbrace  \mathrm{M}_Q w  \geq (A^*/A) \cdot w_Q  \rbrace \right |^{1/p'} .
\end{align}
Using first Markov's Inequality and then the Fujii-Wilson \eqref{eq:definitionofFujiiWilson} condition, one has
$$
\left | Q \cap \lbrace  \mathrm{M}_Q w  \geq (A^*/A) \cdot w_Q  \rbrace \right | \leq \frac{A}{A^*\cdot w_Q} \int_Q \mathrm{M}_Q w \leq   \frac{A}{A^*\cdot w_Q} L \cdot w(Q) = \frac{A \cdot L}{A^*} |Q|.
$$
Plugging this estimate into \eqref{eq:sublevelsetswMQwlongestimates2}, we deduce
\begin{align}\label{eq:sublevelsetswMQwlongestimates3}
w     ( \lbrace x\in Q \, &  : \, A^* \cdot w_Q \leq w(x) < A \cdot \mathrm{M}_Q w (x) \rbrace ) \nonumber \\[0.5em]
& \leq     A \left( \frac{L}{1-(4 L)/p' } \right)^{1/p} |Q|^{\frac{1}{p}-1} w(Q) \left( \frac{A \cdot L}{A^*} |Q| \right)^{1/p'} = \frac{C(A,L,p)}{(A^*)^{1/p'}}  w(Q),  
\end{align}
for a constant $C(A,L,p)>0$ depending only on $A,L$, and $p$. Now, we choose $A^*$ large enough so that
$$
\frac{C(A,L,p)}{(A^*)^{1/p'}} \leq \frac{1-b}{2}.
$$
From \eqref{eq:sublevelsetswMQwlongestimates1} and \eqref{eq:sublevelsetswMQwlongestimates3} we conclude that
$$
w  \left( \lbrace x\in Q \, : \, w(x) \geq A^* \cdot w_Q \rbrace \right) \leq b w(Q) + \frac{1-b}{2}w(Q) = \frac{1+b}{2} w(Q),
$$
where $(1+b)/2 <1.$ By Lemma \ref{lem:equivalencBinftysuperlevelsets}, this shows that $w\in \mathrm{B}_\infty(\D).$
\end{proof}

\begin{proof}[Proof of Theorem \ref{thm:maintheoremBinftyiffFW}]
It follows by combining Lemmas \ref{lem:BinftyimpliesFWanddoubling}, \ref{lem:BinftyimpliesMAbcondition}, and \ref{lem:MAbconditioandFwimplyBinfty}. 
\end{proof}

With an almost identical proof, one can obtain the following equivalence.
\begin{corollary}\label{cor:Binfty+iffMAb+andFW}
For a weight $w$ in $\D,$ the following statements are equivalent.
\begin{itemize}
    \item[(a)] $w$ satisfies \eqref{eq:definitionofFujiiWilson} and \eqref{eq:conditionFWiffRHIForeverythereexists}.
    \item[(b)] For every $\beta \in (0,1)$ there exists $\alpha \in (0,1)$ so that if $Q$ is a Carleson box and $E\subset Q$ is measurable with $|E| \leq \alpha |Q|,$ then $w(E) \leq \beta w(Q).$ 
\end{itemize}
\end{corollary}
\begin{proof}
If $(a)$ holds, it is clear from the proof of Lemma \ref{lem:MAbconditioandFwimplyBinfty} that $w$ satisfies the following property:
\begin{equation}\label{eq:BinftywithditributionalandForeverythereis}
\text{for any } \: \beta \in (0,1) \: \text{ there is } \: \alpha \in (0,1) \: \text{ with } \:     w \left( \lbrace x \in Q\, : \, w(x) \geq \frac{1}{\alpha} w_Q \rbrace \right) \leq \beta w(Q)
\end{equation}
for all Carleson boxes $Q$. This is known to be equivalent to $(b)$ for arbitrary weights and general bases.

The other direction follows the proofs of Lemmas~\ref{lem:BinftyimpliesMAbcondition} and \ref{lem:MAbconditioandFwimplyBinfty} almost verbatim.
\end{proof}

\subsection{Characterizing the Reverse H\"older Inequality in terms of \eqref{eq:definitionofFujiiWilson}}

\begin{lemma}\label{lem:RHIimpliesepsiloncondition}
Let $w$ be a weight in $\D$ satisfying \eqref{eq:definitionofRHI} with exponent $p_0>1,$ that is,
$$
  \vint_Q w^{p_0}  \leq C \left( \vint_Q w \right)^{p_0}, 
$$
for a constant $C\geq 1$ and all Carleson squares $Q$. Then, for every $1 \leq  p \leq  p_0$ and $Q$, we have that
$$
\int_Q w^p \leq C \int_Q (\mathrm{M}_Q w)^p,
$$
with the same constant $C$ as above. That is, condition \eqref{eq:Lpconditionmaximal} holds. 
\end{lemma}
\begin{proof}
Let $1 \leq p \leq p_0.$ Note that the estimate $\mathrm{M}_Q w(x) \geq \vint_Q w$ for all $x\in Q$ gives
$$
\left( \vint_Q \left( \mathrm{M}_Q w \right)^p \right)^{1/p} \geq \left( \vint_Q \left( \vint_Q w \right)^p \right)^{1/p} = \vint_Q w.
$$
Applying first Jensen's Inequality and then \eqref{eq:definitionofRHI}, we get
\begin{align*}
\left( \vint_Q w^p \right)^{1/p} \leq \left( \vint_Q w^{p_0} \right)^{1/p_0} \leq C^{1/p_0} \vint_Q w \leq C^{1/p_0} \left( \vint_Q \left( \mathrm{M}_Q w \right)^p \right)^{1/p},
\end{align*}
which yields the desired inequality.
\end{proof}

\begin{proof}[Proof of Theorem \ref{thm:maintheoremBinftyiffRHI}]
Assume first that $(a)$ holds. If $p_0>1$ is as in \eqref{eq:Lpconditionmaximal}, take $1<p \leq p_0$ smaller than $4 L/(4L-1)$, where $L=[w]_{\FW}.$ Applying condition \eqref{eq:Lpconditionmaximal} for $p$ and Lemma \ref{lemma:selfimprovedyadicmaximal}, we get
$$
\left( \vint_Q w^p \right)^{1/p} \leq C \left( \vint_Q (\mathrm{M}_Q w)^p \right)^{1/p}  \leq  C \left( \frac{L}{1-4L/p' } \right)^{1/p} \vint_Q w .
$$
This is the $p$-reverse H\"older inequality, which can be rewritten as
$$
\vint_Q w^p \leq C^p \frac{L}{1-4L/p' }  \left(  \vint_Q w \right)^p.
$$
Conversely, if $(b)$ holds, then Lemmas \ref{lem:RHIimpliesBinftyanddoub}, \ref{lem:BinftyimpliesFWanddoubling}, and \ref{lem:RHIimpliesepsiloncondition} give $(a).$
\end{proof}

\begin{remark}\label{rem:verificationofLpestimateinmaintheoremRHI}
{\em
In the previous proof, we really only needed the inequality \eqref{eq:Lpconditionmaximal} for a single exponent $p$ satisfying
$$
1<p< \frac{4[w]_{\FW}}{4[w]_{\FW}-1}.
$$ 
}
\end{remark}

\subsection{Characterizing $\Bq$ and $\Blog$ weights in terms of $\Binfty$}

\begin{lemma}\label{lem:BpimpliesLpconditionminimal}
If $w \in \mathrm{B}_{p_0}(\D)$, with $1<p_0<\infty$, then
$$
\int_Q \left( \frac{1}{w} \right)^{\frac{1}{p-1}} \leq [w]_{p_0}^{\frac{1}{p_0 - 1}} \int_Q \left( \frac{1}{\mathrm{m}_Q w} \right)^{\frac{1}{p-1}}
$$
for all Carleson boxes $Q$ and all $p \geq p_0.$ In other words, $w$ satisfies condition \eqref{eq:Lpconditionminimal} with constants $p_0$ and $C= [w]_{p_0}^{\frac{1}{p_0 - 1}}.$
\end{lemma}
\begin{proof}
The $\mathrm{B}_{p_0}$ condition for $w$ is precisely
$$
\left( \vint_Q w \right) \left( \vint_Q w^{\frac{1}{1-p_0}} \right)^{p_0-1} \leq [w]_{p_0},
$$
for all Carleson boxes $Q$, and by Jensen's inequality we get, for all $p \geq p_0$, that
\begin{equation}\label{eq:proofimplicationBpLpestimateminimal}
\left( \vint_Q w \right) \left( \vint_Q w^{\frac{1}{1-p}} \right)^{p-1} \leq [w]_{p_0}.
\end{equation}
Now, by the pointwise estimate $\vint_Q w \geq \mathrm{m}_Q w(x) \Chi_Q(x),$ we get
$$
\left( \vint_Q w \right)^{\frac{1}{1-p}} \leq (\mathrm{m}_Q w)^{\frac{1}{1-p}} \cdot \Chi_Q.
$$
Integrating over $Q$ we obtain
$$
 \left( \vint_Q w \right)^{\frac{1}{1-p}} \leq \vint_Q (\mathrm{m}_Q w)^{\frac{1}{1-p}} ,
$$
and thus
$$
\vint_Q w \geq \left( \vint_Q (\mathrm{m}_Q w)^{\frac{1}{1-p}} \right)^{1-p},
$$
for all boxes $Q.$ This estimate and \eqref{eq:proofimplicationBpLpestimateminimal} lead us to
$$
\left( \vint_Q w^{\frac{1}{1-p}} \right)^{p-1} \leq [w]_{p_0} \left( \vint_Q (\mathrm{m}_Q w)^{\frac{1}{1-p}} \right)^{p-1}  \leq [w]_{p_0}^{\frac{p-1}{p_0-1}} \left( \vint_Q (\mathrm{m}_Q w)^{\frac{1}{1-p}} \right)^{p-1} ,
$$
after using that $[w]_{p_0} \geq 1.$ This is the desired estimate. 
\end{proof}

\begin{lemma}\label{lem:Blogimplieslogminimal}
If $w \in \Blog(\D)$, then
$$
 \int_Q \log \left( \frac{\mathrm{m}_Q w}{[w]_{\log}} \right) \leq \int_Q \log w
$$
for all Carleson boxes $Q.$ Consequently, \eqref{eq:logconditionminimal} holds. 
\end{lemma}
\begin{proof}
By the pointwise estimate $\vint_Q w \geq \mathrm{m}_Q w(x)$ for all $x\in Q$, we get
$$
\log \left( \vint_Q w\right) \geq \log \left( \mathrm{m}_Q w (x) \right), \quad x\in Q.
$$
Integrating over $Q$ we derive that
$$
\log \left( \vint_Q w \right) \geq \vint_Q \log \left( \mathrm{m}_Q w \right),
$$
and thus
$$
\vint_Q w \geq \exp \left( \vint_Q \log \left( \mathrm{m}_Q w \right) \right)
$$
for all boxes $Q.$ This estimate and the definition of the \eqref{eq:defBlog} condition gives us that
$$
 \exp \left( \vint_Q \log \left( \mathrm{m}_Q w \right) \right) \leq \vint_Q w \leq [w]_{\log} \exp \left( \vint_Q \log w \right) .
$$
Taking logarithms and rearranging, we get the desired inequality. 
\end{proof}

Our next goal is to show that conditions \eqref{eq:Lpconditionminimal} (resp. \eqref{eq:logconditionminimal}) are sufficient for a $\Binfty$ weight to satisfy a \eqref{eq:defBp} (resp. \eqref{eq:defBlog}) property. 

\begin{lemma}\label{lem:BinftyimpliesEstimateMinimal}
Let $w \in \Binfty(\D)$, with parameters $\alpha$ and $\beta$. Then there is a constant $C=C(w)$ such that
\begin{equation}\label{eq:keyestimateRHIimpliesBp}
|Q \cap \lbrace 1/ \mathrm{m}_{Q} w > \lambda \rbrace | \leq  \frac{C(w) \lambda}{1-\beta} w \left( Q \cap \lbrace 1/w > \lambda \alpha \rbrace \right), 
\end{equation}
for all Carleson boxes $Q$ and $\lambda > 1/w_{Q}. $
\end{lemma}
\begin{proof}
 For all Carleson square $Q$ denote $E= Q \cap \lbrace w \geq w_{Q}/\alpha \rbrace$. By Markov's inequality, we have that $|E| \leq \alpha |Q|$ and so 
\begin{equation}\label{eq:distributionalvis}
w\left(  Q \cap  \lbrace w \geq w_{Q}/\alpha \rbrace   \right) \leq \beta w(Q).
\end{equation}
Now let $\ud \mu: =w \ud x$, which is a doubling measure by Lemma \ref{lem:RHIimpliesBinftyanddoub}. So, for each box $Q$ and $\lambda > \vint_{Q} w^{-1} \ud \mu =  1/w_{Q}$, let $\lbrace Q_j \rbrace_j$ be the Calderón-Zygmund cubes at height $\lambda$ for the function $w^{-1}$ with $\mu$ as the underlying measure, cf. Lemma \ref{lem:CZdecomposition}. We then have that
\begin{equation}\label{eq:estimateforCZcubes}
\lambda < \frac{1}{w_{Q_j}} \leq C(w) \lambda, \quad \text{for all} \quad j.
\end{equation} 
Combining \eqref{eq:distributionalvis} and \eqref{eq:estimateforCZcubes}, we get
$$
w(Q_j)    = w\left( Q_j \cap \lbrace w \geq w_{Q_j}/\alpha \rbrace \right) + w\left( Q_j \cap \lbrace w < w_{Q_j}/\alpha \rbrace \right)  \leq \beta w(Q_j) + w\left( Q_j \cap \lbrace w < 1/ (\lambda \alpha) \rbrace \right),
$$
and thus
$$
w(Q_j) \leq \frac{1}{1-\beta} w\left( Q_j \cap \lbrace w < 1/ (\lambda \alpha) \rbrace \right).
$$
This estimate, along with \eqref{eq:estimateforCZcubes} and the properties of the Calderón-Zygmund boxes lead us to
\begin{align*}\label{eq:estimatelevelsetsinverseweightedMaximal}
| Q \cap \lbrace \mathrm{M}_Q^\mu (w^{-1}) > \lambda \rbrace | & = \Big | \bigcup_j Q_j \Big |   \leq C(w) \lambda \sum_j w(Q_j) \nonumber \\
& \leq \frac{C(w) \lambda}{1-\beta} \sum_j  w\left( Q_j \cap \lbrace w < 1/ (\lambda \alpha) \rbrace \right) \nonumber \\
& \leq  \frac{C(w) \lambda}{1-\beta} w \left( Q   \cap \lbrace w < 1/ (\lambda \alpha) \rbrace \right).   
\end{align*}
This is the desired inequality, since for all $x\in Q=Q_{I},$ one has
\begin{align*}
\mathrm{M}_{Q}^\mu (w^{-1})(x) & = \sup_{Q_J\ni x, \, J \in \mathcal{D}(I)} \frac{\int_{Q_J} w^{-1} \ud \mu}{\mu(Q_J)} = \sup_{Q_J\ni x, \, J \in \mathcal{D}(I)} \frac{|Q_J|}{w(Q_J)} \\[0.5em]
& = \frac{1}{\underset{Q_J \ni x, \, J \in \mathcal{D}(I)}{\inf} \frac{w(Q_J)}{|Q_J|}} = \frac{1}{\mathrm{m}_{Q} w(x)}. \qedhere
\end{align*}
\end{proof}

\begin{lemma}\label{lem:BinftyandExtraconditionimpliesBpBlog}
Let $w \in \Binfty(\D)$. Then, the following hold.
\begin{itemize}
    \item[(i)] If $w$ satisfies \eqref{eq:Lpconditionminimal}, then $w \in \Bq(\D)$ for some $1 \leq q <\infty. $
    \item[(ii)] If $w$ satisfies \eqref{eq:logconditionminimal}, then $w$ belongs to $\Blog(\D).$ 
\end{itemize}
\end{lemma}
\begin{proof}
We begin with part $(i)$, for which we assume \eqref{eq:Lpconditionminimal}, and let $p_0$ be as in that condition. Letting $\varepsilon_0: = 1/(p_0-1)$ and $0< \varepsilon \leq \varepsilon_0$, condition \eqref{eq:Lpconditionminimal} and Lemma \ref{lem:BinftyimpliesEstimateMinimal} yield, for an arbitrary Carleson square $Q$, and a constant $C>0,$
\begin{align*}
\int_{Q} w^{-\varepsilon}  & \leq C \int_{Q} (\mathrm{m}_{Q} w)^{-\varepsilon} = C \int_0^\infty \varepsilon \lambda^{\varepsilon-1} |Q \cap \lbrace 1/ \mathrm{m}_Q w > \lambda \rbrace | \ud \lambda  \\
& \leq C \int_0^{1/w_{Q}} \varepsilon \lambda^{\varepsilon-1} |Q| \ud \lambda + C \int_{1/w_{Q}}^\infty \varepsilon \lambda^{\varepsilon-1} |Q \cap \lbrace 1/ \mathrm{m}_Q w > \lambda \rbrace | \ud \lambda \\
& \leq C \frac{|Q|}{(w_Q)^{\varepsilon}} + C \int_{0}^\infty \varepsilon \lambda^{\varepsilon} w \left( Q \cap \lbrace 1/w >\lambda \alpha \rbrace \right) \ud \lambda \\
& =  C \frac{|Q|}{(w_{Q})^{\varepsilon}} + \frac{  C \varepsilon}{(1+\varepsilon) \alpha^{1+\varepsilon}} \int_{Q} w^{-\varepsilon}.
\end{align*}
Dividing by $|Q|$, we deduce that
$$
\vint_{Q} w^{-\varepsilon} \leq C (w_Q)^{-\varepsilon} + \frac{  C \varepsilon}{(1+\varepsilon) \alpha^{1+\varepsilon}} \vint_Q w^{-\varepsilon}.
$$
Now, condition \eqref{eq:Lpconditionminimal} implies that $w^{-\varepsilon_0} \in L^1(\D),$ and so, by H\"older's Inequality, also $w^{-\varepsilon} \in L^1(\D)$. For sufficiently small $\varepsilon>0$ we thus conclude that 
$$
\left( \vint_{Q} w \right)^{\varepsilon} \vint_{Q} w^{-\varepsilon} \leq \frac{C}{1-\frac{ C \varepsilon}{(1+\varepsilon) \alpha^{1+\varepsilon}} }.
$$
This shows that $w\in \Bq(\D)$ for $q= 1+ 1/\varepsilon.$

For part $(ii)$, assume that $w$ satisfies \eqref{eq:logconditionminimal}. If $\int_Q \log w = +\infty,$ \eqref{eq:defBlog} trivially holds. We may therefore assume $\int_Q \log w \in \R.$ Using condition \eqref{eq:logconditionminimal}, we find a constant $c>0$ such that, for all boxes $Q,$
\begin{align*}
   \int_Q \log & \left( \frac{c \cdot w_Q}{w} \right)      \leq    \int_Q \log \left( \frac{w_Q}{\mathrm{m}_Q w} \right)   \\[0.5em] 
   & = \int_1^\infty \frac{1}{t} | Q \cap \lbrace 1/\mathrm{m}_Q w(x) \geq t/w_Q \rbrace |  \ud t 
    =\int_{1/w_Q}^\infty \frac{1}{t} |Q \cap \lbrace 1/\mathrm{m}_Q w >t \rbrace | \ud t.
\end{align*}

By Lemma \ref{lem:BinftyimpliesEstimateMinimal}, the last term is not greater than
\begin{align*}
\frac{C(w)}{1-\beta} \int_{1/w_Q}^\infty w  & \left( Q \cap \lbrace  1/ w > \alpha t \rbrace \right) \ud t  
 \leq   \frac{C(w)}{1-\beta} \int_{0}^\infty w \left( Q \cap \lbrace  1/ w > \alpha t \rbrace \right) \ud t \\[0.5em]
& = \frac{C(w)}{(1-\beta)\alpha} \int_{0}^\infty w \left( Q \cap \lbrace  1/ w >   t \rbrace \right) \ud t  =  \frac{C(w)}{(1-\beta)\alpha} \int_Q \frac{1}{w} \ud w  = \frac{C(w)}{(1-\beta)\alpha} |Q|.
\end{align*}
Here $\alpha$ and $\beta$ are the \eqref{eq:defBifnty}-constants of $w$. Dividing by $|Q|$, there is thus a constant $C' = C'(w)$ such that
$$
  \vint_Q \log   \left( \frac{c \cdot w_Q}{w} \right)   \leq C',
$$
which yields that $w \in \Blog(\mathbb{D})$.
\end{proof}

Combining the lemmatas in this section, we have now proven Theorems~\ref{thm:maintheoremBp} and \ref{thm:maintheoremBlog}. Similarly to Remark~\ref{rem:verificationofLpestimateinmaintheoremRHI}, the proof of Theorem~\ref{thm:maintheoremBp} shows that it is sufficient to check the inequality \eqref{eq:Lpconditionminimal} for one sufficiently large $p$, depending on the parameters of the \eqref{eq:defBifnty}-condition and the doubling constant of $w$. We choose not to make this explicit.

\section{Further results and examples}\label{sec:examples}

\subsection{$\Bp$-weights and bounded mean oscillation}
In the classical setting, one obtains the functions of bounded mean oscillation by taking logarithms of Muckenhoupt weights. This is a consequence of the celebrated John--Nirenberg inequality \cite{JN61}. 
The connection between Békollé--Bonami weights and $\mathrm{BMO}(\mathbb{D})$ was explored in \cite{LN23}, with bounded hyperbolic oscillation as a key side condition.

In general it is natural to consider the larger class of functions $f \in \mathrm{BMO}_{\mathrm{C}}(\mathbb{D})$ for which
$$
\|f\|_{\mathrm{BMO}_{\mathrm{C}}}:= \sup_{I \text{ arc of } \T} \vint_{Q_I} |f-f_{Q_I}| < \infty,
$$
where the supremum is taken only over the basis Carleson squares. Similarly, we say that $f\in L^1(\D)$ satisfies the John--Nirenberg inequality with respect to Carleson squares when there are constants $C$ and $b >0$ such that 
\begin{equation}\label{eq:JNineq}
|\lbrace x\in Q \, : \,  |f(x)-f_Q| >t \rbrace | \leq C \exp(-bt) |Q| \tag{JN}   \end{equation}
for all $t>0$ and all Carleson squares $Q.$

We outline the proof of the following characterization of logarithms of $\Bp$-weights. Note that a side condition is necessary, since $f(z) = -\frac{1}{|z|}$, $z \in \mathbb{D}$, is an example of a function $f \in \mathrm{BMO}_{\mathrm{C}}(\mathbb{D})$ such that $e^{\varepsilon f} \in L^1(\mathbb{D})$ for every $\varepsilon > 0$, but $e^{\varepsilon f}$ is never a $\Bp(\mathbb{D})$-weight for any $1 < p < \infty$.

\begin{theorem}\label{thm:BMOcharacterizationJN}
Let $1<p< \infty$ and $f\in  L^1(\D)$. Then the following statements are equivalent.
\begin{itemize}
    \item[\rm (a)] There exists $\varepsilon >0$ such that that $e^{\varepsilon f} \in \mathrm{B}_p(\D).$
    \item[\rm (b)] $f$ satisfies the John-Nirenberg inequality \eqref{eq:JNineq}.
  \item[\rm (c)] $f\in \mathrm{BMO}_{\mathrm{C}}(\D)$ and there are constants $C,A,a > 0$ such that
\begin{equation}\label{eq:conditionMexp}
| \lbrace x\in Q \,  : \,  |f(x)-f_Q|-A \mathrm{M}_Q(f-f_Q)(x) >t \rbrace | \leq C \exp(-at) |Q| \tag{Mexp}    
\end{equation}
for every $t>0$ and Carleson square $Q$.
\end{itemize}
\end{theorem}
\begin{proof}[Proof sketch]
By scaling, it is sufficient to consider $p = 2$.
It is clear that (b) implies (c). That (a) and (b) are equivalent has an identical proof to the corresponding result for the basis of cubes in $\R^n$, cf. \cite[Corollary 2.18]{GCRF85}.

It remains to show that (c) implies (b). The key is that 
$f\in \mathrm{BMO}_{\mathrm{C}}(\D)$ implies the existence of constants $K,k>0$ so that
\begin{equation}\label{eq:estimateMQJN}
 |\lbrace x\in Q \, : \,  \mathrm{M}_Q(f-f_Q)(x) >t \rbrace | \leq K |Q| \exp \left(- kt/\|f\|_{\mathrm{BMO}} \right) 
\end{equation}
for every $t>0$ and $Q$. This can be proved by imitating the arguments of \cite[pp. 64--66]{N77}, taking Lemma \ref{lem:CZdecomposition} into account and replacing the sets $Q \cap \lbrace |f-f_Q| >t \rbrace$ with $Q \cap \lbrace \mathrm{M}_Q(f-f_Q) >t \rbrace$ at the appropriate points. Clearly, \eqref{eq:estimateMQJN} in conjunction with \eqref{eq:conditionMexp} allows us to derive \eqref{eq:JNineq} for sufficiently small $b > 0$.
\end{proof}

\subsection{Examples} We finish the paper with two examples of weights demonstrating the optimality of our results. First we give an example of a weight which is not dominated by any of its maximal functions. In particular, it falls outside of the frameworks considered in \cite{APR19, B04, DLP25, LN23}.

\begin{example}\label{ex:BpandRHInotBHO} \rm
We shall consider a radial weight $w(z) = v(1-|z|)$, $z \in \mathbb{D}$, where
\[
v(r) = \begin{cases} x^{-k}, &2^{-k-1} < r \leq 2^{-k-1}(1 + 2^{-k-1})\\

k+1, &2^{-k-1}(1 + 2^{-k-1}) < r \leq 2^{-k}(1 - 2^{-k-2}), \\ x^k, &2^{-k}(1 - 2^{-k-2}) < r \leq 2^{-k}, \end{cases}\qquad k = 0, 1, 2, \ldots,
\]
for a given number $1 < x < 2$. Then the following hold:
\begin{enumerate}
    \item $w$ satisfies a reverse H\"older inequality;
    \item  $w$ satisfies \eqref{eq:Lpconditionminimal} and $w \in \mathrm{B}_q(\D)$ for some $1 < q < \infty$;
    \item $w \cdot \Chi_Q \not \lesssim \mathrm{M}_Qw $ and $\mathrm{m}_Qw \not \lesssim w \cdot \Chi_Q$, with constants independent of the Carleson box $Q.$
\end{enumerate}
To see this, note that for any $q \in \mathbb{R}$ and Carleson square $Q$ with 
\[ 2^{-n-1} < \mathrm{dist}(Q, 0) \leq 2^{-n},\]
we have that
\[ \vint_Q w^q \approx 2^n \sum_{k=n}^\infty 2^{-k}(k^q + 2^{-k}(x^{-qk} + x^{qk})) \approx \begin{cases} n^q, &\frac{1}{2} < x^q \leq 2, \\ 
1, &x^q = \frac{1}{2}, \\
(\frac{x}{2})^n, &2 < x^q < 4, \\
(\frac{1}{2x})^n, &\frac{1}{4} < x^q < \frac{1}{2}, \\
\infty, &\textrm{otherwise}.\end{cases}
\]
Therefore $w \in \RHI_p$ if and only if $1 \leq p \leq \frac{\log 2}{\log x}$, and $w \in \mathrm{B}_q(\mathbb{D})$ if and only if $1 + \frac{\log x}{\log 2} < q \leq \infty$. 
In view of Theorems~\ref{thm:maintheoremBinftyiffRHI} and \ref{thm:maintheoremBp}, $w$ must then satisfy \eqref{eq:Lpconditionminimal} and \eqref{eq:Lpconditionmaximal} (and thus \eqref{eq:logconditionminimal}). This can also be verified directly.
\end{example}

Next we give an example of a very poorly behaved weight with the Fujii-Wilson condition, demonstrating that side conditions are a must. Furthermore, the example shows that one cannot replace the maximal function $\mathrm{M}_Q w$ with the general maximal function $\mathrm{M} w$ in condition \eqref{eq:Lpconditionmaximal}.

\begin{example}  \rm
We again consider a radial weight $w\colon \mathbb{D} \to [0,\infty),$
$$ w(re^{i\theta}) = 
\begin{cases} \frac{v(r)}{r}, \; r > 1/2, \\[0.6em]
2v(1/2), \; r \leq 1/2,
\end{cases}
$$
where
$$
v(x) = \frac{2 e^{-\frac{1}{(1-x)^2}}}{ (1-x)^3},  \quad 0<x<1. 
$$ 
The following hold.
\begin{enumerate}
    \item $w\in L^\infty(\D)$, so that, in particular $w$ is a weight in $\D.$
    \item $w$ satisfies the Fujii-Wilson condition.
    \item $w(x) \lesssim \mathrm{M} w(x) $ for all $x\in \D,$ but $w \cdot \Chi_Q \not \lesssim \mathrm{M}_Qw$, with a constant independent of $Q$.
    \item $w$ is not doubling. In particular, $w$ does not satisfy the \eqref{eq:defBifnty} property, and thus no reverse H\"older inequality holds for $w$. Moreover, in accordance with Theorem \ref{thm:maintheoremBinftyiffFW}, $w$ cannot satisfy \eqref{eq:conditionforFWiffBinfty}.
\end{enumerate}
\end{example}

Observe that the function
\begin{equation}\label{eq:increasingaverages}
(0,1) \ni x \mapsto \frac{e^{-\frac{1}{(1-x)^2}}}{1-x} 
\end{equation}
is decreasing. Therefore, for $|z| > 1/2$,
$$
\mathrm{M} w (z) \approx \sup_{1/2 \leq a \leq |z| < 1} \frac{\int_a^1 v}{1-a} = \sup_{1/2 \leq a \leq |z| < 1} \frac{e^{-\frac{1}{(1-a)^2}}}{1-a} = 2e^{-4}.
$$
Since $w$ is bounded, we conclude that $w(z) \lesssim \mathrm{M} w(z)$, $z \in \mathbb{D}$.

Suppose now that $Q$ is a Carleson square with $\mathrm{dist}(Q, 0) = b \geq 1/2$.
Then, for $z \in Q$, using again that the function in \eqref{eq:increasingaverages} is decreasing, we see that
\begin{align*}
\M(w \cdot \Chi_Q)(z) &\approx \sup_{1/2 \leq a \leq |z| < 1} \frac{\int_a^1 v \cdot \Chi_{[b,1)}}{1-a} =\sup_{1/2 \leq a \leq |z| < 1} \frac{e^{-\frac{1}{(1-\max(a,b))^2}} }{1-a} \\
& \leq  \sup_{1/2 \leq a \leq x < 1} \frac{e^{-\frac{1}{(1-b)^2}} }{1-b} \frac{1-\max(a,b)}{1-a} \leq  \frac{e^{-\frac{1}{(1-b)^2}} }{1-b}.
\end{align*}
This shows that $w \cdot \Chi_{Q} \not \lesssim \M(w \cdot \Chi_Q)$,
as otherwise there would exist a constant $C>0$ such that
$$
\frac{2 e^{-\frac{1}{(1-|z|)^2}}}{ (1-|z|)^3} \leq C \frac{e^{-\frac{1}{(1-b)^2}} }{1-b}, \quad  \quad |z| \in (b,1), \; b\in [1/2,1),
$$
which would imply that $(1-b)^2 \geq 1/(2C)$ for all $b \in [1/2,1)$. In particular, $w \cdot \Chi_Q \not \lesssim \mathrm{M}_Q w$.

The same computation shows that $w$ satisfies the Fujii-Wilson condition, since
\[\M(w \cdot \Chi_Q)(z) \lesssim \frac{e^{-\frac{1}{(1-b)^2}} }{1-b} \approx \vint_{Q} w, \qquad z \in Q, \, \textrm{dist}(Q,0) =b \geq 1/2.\]
and consequently
\[
\int_{Q} \M (w \cdot \Chi_{Q}) \lesssim \int_{Q} w.
\]

Finally, to check that $w$ is not a doubling weight, note that for $b \in (1/2,1)$ we have
$$
\frac{w(Q)}{w(\frac{1}{2}Q)} \approx \dfrac{e^{-\frac{1}{(1-b)^2}}}{e^{-\frac{1}{((1-b)/2)^2}}} = \dfrac{e^{-\frac{1}{(1-b)^2}}}{e^{-\frac{4}{(1-b)^2}}} = e^{\frac{3}{(1-b)^2}},
$$
which diverges as $b \to 1^-.$

\section*{Acknowledgements}

Both authors were supported by grant no. 334466 of the Research Council of Norway, ``Fourier Methods and Multiplicative Analysis''.

\bibliographystyle{amsplain} 
\bibliography{Bpissues} 
\end{document}